\documentclass[12pt]{article}
\usepackage{fullpage}
\usepackage{amssymb,amsmath,amsfonts,amsthm,latexsym,enumerate,url,cases}
\numberwithin{equation}{section}
\usepackage{hyperref}
\hypersetup{colorlinks=true,citecolor=blue,linkcolor=blue,urlcolor=blue}

\newtheorem{theorem}{Theorem}[section] %
\newtheorem{lemma}[theorem]{Lemma} %
\newtheorem{remark}[theorem]{Remark} %

\begin{document}
\title{Asymptotic formulas \\ for general colored partition
functions}
\author{  Yong-Gao Chen\footnote{Corresponding author, ygchen@njnu.edu.cn(Y.-G. Chen)}, Ya-Li Li\\
\small  School of Mathematical Sciences and Institute of Mathematics, \\
\small  Nanjing Normal University,  Nanjing  210023,  P. R. China
}
\date{}
\maketitle \baselineskip 18pt

\begin{abstract} In 1917, Hardy and Ramanujan obtained the asymptotic formula for the classical partition
function $p(n)$. The classical partition function $p(n)$ has been
extensively studied. Recently, Luca and Ralaivaosaona obtained the
asymptotic formula for the square-root function. Many
mathematicians have paid much attention to congruences on some
special colored partition functions. In this paper, we investigate
the  general colored partition functions. Given positive integers
$1=s_1<s_2<\dots <s_k$ and $\ell_1, \ell_2,\dots , \ell_k$. Let
$g(\mathbf{s}, \mathbf{l}, n)$  be the number of
 $\ell$-colored partitions of $n$ with
$\ell_i$ of the colors appearing only in multiplies of $s_i\ (1\le
i\le k)$, where $\ell = \ell_1+\cdots +\ell_k$. By using the
elementary method we obtain  an asymptotic formula for the
partition function $g(\mathbf{s}, \mathbf{l}, n)$ with an explicit
error term.
\end{abstract}



\section{Introduction}

Let $p(n)$ denote the number of partitions of $n$, i.e.
$$n=a_1+\cdots+a_k$$with integers $1\le a_1\le \cdots \le
a_k.$  The generating function of $p(n)$ is
\begin{equation}\label{p(3)}f(z)=1+\sum_{n=1}^{\infty}p(n)z^n=\prod_{n=1}^\infty
\frac1{1-z^n}.\end{equation} Ramanujan \cite{Ramanujan} obtained
 many  congruent identity for $p(n)$. Hardy and
Ramanujan \cite{Hardy1917} and Uspensky \cite{Uspensky}
independently proved that\begin{equation}\label{p(1a)}p(n)\sim
\frac 1{4n\sqrt 3} \exp \left( \pi \sqrt{\frac{2n}3}\right).
\end{equation} An elementary proof for this formula is given
by Erd\H os \cite{Erdos} with no explicit constant $(4\sqrt
3)^{-1}$. Lehmer \cite{Lehmer} gave the series for the partition
function $p(n)$. Odlyzko \cite[(1.6)]{Odlyzko} gave an asymptotic
formula for $p(n)$ with an explicit error term. That is,
\begin{equation}\label{p(2a)}p(n)=\frac {1+O(n^{-1/2})}{4n\sqrt 3}
\exp \left( \pi \sqrt{\frac{2n}3}\right).
\end{equation}

The partition function $p(n)$ has a long history and generates
varieties.

Recently, Luca and Ralaivaosaona \cite{Luca2016} obtained the
asymptotic formula for the square-root function $q(n)$, which is
defined to be the number of solutions of
$$ n=[\sqrt{a_1}]+[\sqrt{a_2}]+\cdots +[\sqrt{a_k}]$$
with integers $1\le a_1\le a_2\le \cdots \le a_k$. For related
results, one may refer to Balasubramanian and Luca \cite{Balu} and
Chen and Li \cite{ChenCR} and \cite{ChenTaiwan}.

Let $p_k(n)$ be the number of 2-color partitions of $n$ where one
of the colors appears only in parts that are multiples of $k$. The
generating function of $p_k(n)$ is
$$1+\sum_{n=1}^\infty p_k(n)z^n =\prod_{n=1}^\infty
\frac1{(1-z^n)(1-z^{kn})}.$$
 Chan \cite{Chan1}, Kim \cite{Kim} and Sinick
\cite{Sinick} studied some results of the case $k=2$. Recently,
Ahmed, Baruah and Dastidar \cite {Ahmed} and Chern \cite{Chern}
obtained many congruences of $p_k(n)$ for some $k$.

By further analogy, Chan and Cooper \cite{Chan and Cooper} and
Chen \cite{Chen} considered a special partition function $c(n)$
which is the number of 4-colored partitions of $n$ with two of the
colors appearing only in multiplies of 3. The generating function
of $c(n)$  is
\begin{eqnarray*}1+\sum_{n=1}^\infty c(n)z^n
=\prod_{n=1}^\infty \frac1{(1-z^n)^2(1-z^{3n})^2} .
\end{eqnarray*}

In this paper, we focus on  the asymptotic formula for the general
colored partition functions.

Given integers $1=s_1<s_2<\dots <s_k$ and $\ell_1, \ell_2,\dots ,
\ell_k$. Let $g(\mathbf{s}, \mathbf{l},n)$ be the number of
$\ell$-colored partitions of $n$ with $\ell_i$ of the colors
appearing only in multiplies of $s_i\ (1\le i\le k)$, where $\ell
=\ell_1+\cdots +\ell_k$. Write
$$\mathbf{s}=(s_1,s_2,\dots , s_k)$$
and $$ \mathbf{l} =(\ell_1, \ell_2,\dots , \ell_k).$$We call
$g(\mathbf{s}, \mathbf{l}, n)$ the $(\mathbf{s},
\mathbf{l})$-colored partition function. For convenience, we
define $g(\mathbf{s}, \mathbf{l}, 0)=1$ and $g(\mathbf{s},
\mathbf{l}, n)=0$ for all $n<0$. The generating function of
$g(\mathbf{s}, \mathbf{l}, n)$  is
\begin{eqnarray}\label{g}\sum_{n=0}^\infty g(\mathbf{s},
\mathbf{l}, n) x^n=\prod_{n=1}^\infty \frac 1{
(1-z^{s_1n})^{\ell_1} \cdots (1-z^{s_kn})^{\ell_k}}.\end{eqnarray}

In this paper, the following result is proved.

\begin{theorem} \label{thm0} For any given $\varepsilon >0$
(small), we have
$$g(\mathbf{s}, \mathbf{l }, n)=c(\mathbf{s}, \mathbf{l}) n^{d(\mathbf{l} )}
\exp \left(\pi \sqrt{\frac{2a (\mathbf{s}, \mathbf{l})
n}3}\right)+O\left(n^{d(\mathbf{l} )-\frac 14 +\varepsilon }\exp
\left(\pi \sqrt{\frac{2a (\mathbf{s}, \mathbf{l}) n}3}\right)
\right) ,$$ where
$$a (\mathbf{s}, \mathbf{l})=\sum_{i=1}^k\frac{\ell_i}{s_i},\quad d(\mathbf{l} )= -\frac 34 -\frac 14 (\ell_1+\cdots
+\ell_k),$$
$$c(\mathbf{s}, \mathbf{l}) = 2^{-(3\ell_1+\cdots
+3\ell_k+5)/4} 3^{-(\ell_1+\cdots +\ell_k+1)/4} a(\mathbf{s},
\mathbf{l})^{(\ell_1+\cdots +\ell_k+1)/4} s_1^{\ell_1/2}\cdots
s_k^{\ell_k /2}.$$
\end{theorem}

\begin{remark}
By employing the Tauberian theorem of Ingham \cite{Ingham}, under
its form in the ¡°special case¡±, it is possible to get an
asymptotic formula for $g(\mathbf{s}, \mathbf{l }, n)$ without the
error term. We do not intend to give the details here.

Let $k=2,s_1=1, s_2=3$ and $\ell_1=\ell_2=2$. By
 Theorem \ref{thm0}, we can give an asymptotic formula for
$c(n)$, that is,
$$c(n)= \frac{1}{3 \sqrt 6 n^{7/4}} \exp \left( \frac 43 \pi  \sqrt n\right)
+O\left(n^{-2+\varepsilon}\exp \left( \frac 43 \pi  \sqrt
n\right)\right).$$

We believe that there are many congruences for the $(\mathbf{s},
\mathbf{l})$-colored partition functions as many known various
partition functions.
\end{remark}

\section{The main ingredients}

For convenience, let
\begin{equation}\label{acc}a=a (\mathbf{s}, \mathbf{l})=\sum_{i=1}^k\frac{\ell_i}{s_i}, \quad g(n)=g(\mathbf{s}, \mathbf{l }, n),
\quad c_1=\pi \sqrt{\frac{2}3},\quad c_2=\frac 1{4\sqrt
3}.\end{equation} Then \eqref{p(2a)} becomes
\begin{equation}\label{p(1)}p(n)= \frac {c_2+O(n^{-1/2})}{n} \exp
\left( c_1 \sqrt{n}\right).
\end{equation}
If $k=1$ and $\ell_1=1$, then $g(n)=p(n)$. In this case, Theorem
\ref{thm0} follows from \eqref{p(1)}. Now we assume that
$k+\ell_1\ge 3$. So $as_k>1$.

\begin{lemma}\label{lem1}We have
$$g(n)=\sum_{\mathbf{u}\in U_n}
\prod_{\substack{1\le i\le k\\ 1\le j\le \ell_i}}  p(u_{i,j}),$$
where $U_n$ is the set of all $\ell_1+\cdots +\ell_k$ tuples
$$\mathbf{u}= (u_{1,1} , \dots , u_{1, \ell_1} , \dots , u_{k,1}
, \dots , u_{k, \ell_k}) $$ of nonnegative integers with
$$s_1u_{1,1}+\cdots +s_1 u_{1, \ell_1} +\cdots +s_ku_{k,1}+\cdots
+s_k u_{k, \ell_k} =n.$$ \end{lemma}

\begin{proof}By \eqref{p(3)} and \eqref{g}, we have\begin{eqnarray*}
\sum_{n=0}^\infty g(n)z^n&=&\Big( \sum_{n=0}^\infty
p(n)z^{s_1n}\Big)^{\ell_1}\cdots\Big(\sum_{n=0}^\infty p(n)z^{s_kn}\Big)^{\ell_k}\\
&=&\sum_{u_{1,1}=0}^\infty p(u_{1,1})z^{s_1u_{1,1}}\cdots\sum_{u_{1,\ell_1}=0}^\infty p(u_{1,\ell_1})z^{s_1u_{1,\ell_1}}\cdots\\
&&\sum_{u_{k,1}=0}^\infty
p(u_{k,1})z^{s_ku_{k,1}}\cdots\sum_{u_{k,\ell_k}=0}^\infty
p(u_{k,\ell_k})z^{s_ku_{k,\ell_k}}
\\&=&\sum_{n=0}^\infty \Big(\sum_{\mathbf{u}\in U_n}
\prod_{\substack{1\le i\le k\\ 1\le j\le \ell_i}} p(u_{i,j})\Big)
z^{n}.
\end{eqnarray*}
Now Lemma \ref{lem1} follows immediately.
\end{proof}

\medspace

We will divide $U_n$ into two parts $U_n'$ and $U_n''$ which will
be given later such that $u_{i,j}\to +\infty$ as $n\to +\infty$
for any $\mathbf{u}= (u_{1,1} , \dots , u_{1, \ell_1} , \dots ,
u_{k,1} , \dots , u_{k, \ell_k})\in U_n'$. By Lemma \ref{lem1}, we
have
\begin{equation}\label{eqg(n)} g(n)=\sum_{\mathbf{u}\in U_n'}\prod_{\substack{1\le i\le k\\ 1\le j\le \ell_i}}
p(u_{i,j})+\sum_{\mathbf{u}\in U_n''}\prod_{\substack{1\le i\le k\\
1\le j\le \ell_i}} p(u_{i,j}).\end{equation}

We expect that $U_n'$ contributes to $g(n)$  the main term and
$U_n''$ contributes to $g(n)$ the remainder term.

 By \eqref{p(1)}, we have
 \begin{eqnarray}&&\sum_{\mathbf{u}\in U_n'}\prod_{\substack{1\le i\le k\\ 1\le j\le \ell_i}}
p(u_{i,j})\nonumber \\
&=&\nonumber\sum_{\mathbf{u}\in
U_n'}\prod_{\substack{1\le i\le k\\ 1\le j\le \ell_i}}\frac
{c_2+O(u_{i,j}^{-1/2})}{u_{i,j}} \exp
\left(c_1\sqrt{u_{i,j}}\right) \nonumber \\
&=&\label{eqw1}\sum_{\mathbf{u}\in U_n'} \frac{c_2^{\ell_1+\cdots
+\ell_k}+O\big(\sum u_{i,j}^{-1/2}\big)}{\prod u_{i,j}}\exp \Big(
c_1 \sum_{i,j} \sqrt{u_{i,j}}\Big) \end{eqnarray} and
\begin{eqnarray}\label{eqw2}\sum_{\mathbf{u}\in U_n''}\prod_{\substack{1\le i\le k\\ 1\le j\le \ell_i}}
p(u_{i,j}) \ll \sum_{\mathbf{u}\in U_n''} \exp \Big( c_1
\sum_{\substack{1\le i\le k\\ 1\le j\le \ell_i}}
\sqrt{u_{i,j}}\Big).\end{eqnarray}

Since the function $\exp (x)$ increases rapidly, it infers from
\eqref{eqw1} and \eqref{eqw2} that the maximal value of
$\sum\limits_{i,j}\sqrt{u_{i,j}}$ with $ \mathbf{u}\in U_n$ gives
the main contribution to $g(n)$. Now we find the maximal value of
$\sum\limits_{i,j}\sqrt{u_{i,j}}$ with $ \mathbf{u}\in U_n$.  By
the Cauchy-Schwarz inequality, for any $\mathbf{u}\in U_n$, we
have
\begin{eqnarray*}\sum_{\substack{1\le i\le k\\ 1\le j\le \ell_i}} \sqrt{u_{i,j}}
&=& \sum_{\substack{1\le i\le k\\ 1\le j\le \ell_i}}\frac
1{\sqrt{s_i}}
\sqrt{s_iu_{i,j}}\\
&\le &\Big(\sum_{\substack{1\le i\le k\\ 1\le j\le \ell_i}}\frac
1{s_i}\Big)^{1/2}
\Big(\sum_{\substack{1\le i\le k\\ 1\le j\le \ell_i}}s_iu_{i,j}\Big)^{1/2}\\
&=&\Big(\sum_{i=1}^k\frac{\ell_i}{s_i}\Big)^{1/2}
\sqrt{n},\end{eqnarray*} where the equality holds if and only if
$$u_{i,j}=\frac{n}{s_i^2 a}, \quad 1\le i\le k, 1\le j\le \ell_i
,$$ where
$$a=\sum_{i=1}^k\frac{\ell_i}{s_i}.$$
Let
$$v_{i,j}=\frac{n}{s_i^2 a}, \quad 1\le i\le k, 1\le j\le \ell_i
.$$ It is clear that
$$\sum_{i,j} s_iv_{i,j} =\sum_{i=1}^k \frac{\ell_i n}{s_i a} =n.$$
Now we have proved that the maximal value of
$\sum\limits_{i,j}\sqrt{u_{i,j}}$ is obtained if and only if
$$\mathbf{u}= (u_{1,1} , \dots , u_{1, \ell_1} , \dots , u_{k,1} ,
\dots , u_{k, \ell_k})=(v_{1,1} , \dots , v_{1, \ell_1} , \dots ,
v_{k,1} , \dots , v_{k, \ell_k}):=\mathbf{v}.$$

Basing on the above intuition  that the maximal value of
$\sum\limits_{i,j}\sqrt{u_{i,j}}$ with $ \mathbf{u}\in U_n$ gives
the main contribution to $g(n)$, we take $U_n'$ to be the set of $
\mathbf{u}\in U_n$ for which every $u_{i,j}$ is near to $v_{i,j}$
and $U_n''$ to be the set of $ \mathbf{u}\in U_n$ for which
$u_{i,j}$ is far from $v_{i,j}$ for some pair $i,j$. Since
$$s_1u_{1,1}+\cdots +s_1 u_{1, \ell_1} +\cdots +s_ku_{k,1}+\cdots
+s_k u_{k, \ell_k} =n,$$ it is enough to take $U_n'$ to be those $
\mathbf{u}\in U_n$ for which  $u_{i,j}$ is near to $v_{i,j}$ for
all $i,j$ with $i+j>2$ and $U_n''$ to be those $ \mathbf{u}\in
U_n$ for which $u_{i,j}$ is far from $v_{i,j}$ for some pair $i,j$
with $i+j>2$. Now we give explicit  $U_n'$ and $U_n''$.

We appoint a real number $\eta$ such that
$$\frac 34 <\eta <\min\left\{\frac 56,\ \frac 34 \frac{\ell_1+\cdots +\ell_k-1}{\ell_1+\cdots
+\ell_k-2}\right\}.$$ Since the right hand side is more than
$3/4$, it follows that  such $\eta$ exists. Since $3/4<\eta<5/6$,
we have $2\eta -3/2>0$ and $3\eta -5/2<0$.

Let
$$U_n'=\{ \mathbf{u}\in U_n : |u_{i,j}-v_{i,j}|< v_{i,j}^{\eta} \text{
for all pairs } i, j \text{ with } i+j>2 \} $$ and
$$U_n''=\{ \mathbf{u}\in U_n : |u_{i,j}-v_{i,j}|\ge
v_{i,j}^{\eta} \text{ for some pair }  i, j \text{ with } i+j>2
\}.$$

We hope that $U_n'$ contributes to $g(n)$ the main term and
$U_n''$ contributes to $g(n)$ the remainder term. These will be
proved in the next section.

\section{Preliminary Lemmas}

Firstly we prove that $U_n''$ contributes to $g(n)$ the remainder
term.

\begin{lemma}\label{lem6}We have$$\sum_{\mathbf{u}\in U_n''}\prod_{\substack{1\le i\le k\\
1\le j\le \ell_i}} p(u_{i,j})= O\left( \exp \left(
c_1\sqrt{an}-c_3n^{2\eta -3/2}\right) \right), $$where $c_1$ is
given by \eqref{acc} and $c_3$ is a positive constant.
\end{lemma}

\begin{proof}Let $\mathbf{u}\in U_n''$. Without loss of generality, we assume that
\begin{equation}\label{eq4}|u_{k,\ell_k}-v_{k,\ell_k}|\ge
v_{k,\ell_k}^{\eta}.\end{equation} Let
$$u_{k,\ell_k}=v_{k,\ell_k}+\alpha v_{k,\ell_k}.$$
Then $|\alpha| \ge v_{k,\ell_k}^{-(1-\eta)}$. Since $\mathbf{u}\in
U_n''$, it follows that $n\ge s_k u_{k,\ell_k}$. Noting that
$n=s_k^2 a v_{k,\ell_k}$, we have
$$v_{k,\ell_k}+\alpha v_{k,\ell_k}=u_{k,\ell_k}\le \frac
n{s_k}=s_kav_{k,\ell_k}.$$ It follows that $1+\alpha \le s_ka$. So
$(s_ka-1)^{-1} \alpha\le 1$.

By the Cauchy-Schwarz inequality,  we have
\begin{eqnarray*}\sum_{\substack{1\le i\le k\\ 1\le j\le \ell_i}} \sqrt{u_{i,j}}
&=& \sqrt{u_{k,\ell_k}}+\sum_{\substack{1\le i\le k\\ 1\le j\le
\ell_i, (i,j)\not= (k,\ell_k)}}\frac 1{\sqrt{s_i}}
\sqrt{s_iu_{i,j}}\\
&\le &\sqrt{u_{k,\ell_k}}+\Big(\sum_{\substack{1\le i\le k\\ 1\le
j\le \ell_i, (i,j)\not= (k,\ell_k)}}\frac 1{s_i}\Big)^{1/2}
\Big(\sum_{\substack{1\le i\le k\\ 1\le j\le \ell_i, (i,j)\not= (k,\ell_k)}}s_iu_{i,j}\Big)^{1/2}\\
&=&\sqrt{u_{k,\ell_k}}+\Big( a-\frac
1{s_k}\Big)^{1/2} \sqrt{n-s_ku_{k,\ell_k}}\\
&=& \big(v_{k,\ell_k}+\alpha
v_{k,\ell_k}\big)^{1/2}\\
&& +\Big(a-\frac 1{s_k}\Big)^{1/2} \big(
s_k^2av_{k,\ell_k}-s_kv_{k,\ell_k}-\alpha s_kv_{k,\ell_k} \big)^{1/2}\\
&=& f(\alpha ) \sqrt{v_{k,\ell_k}},
\end{eqnarray*}
where
$$f(x) =(1+x
)^{1/2}+(s_ka-1) \left( 1-(s_ka-1)^{-1}x \right)^{1/2}.$$ Since
$$f'(x) =\frac 12 (1+x
)^{-1/2}-\frac 12\left( 1-(s_ka-1)^{-1}x \right)^{-1/2},$$ it
follows that $f'(x)>0$ for $x<0$ and $f'(x)<0$ for $x>0$. By
$|\alpha| \ge v_{k,\ell_k}^{-(1-\eta)}$, we have
$$f(\alpha )\le \max\{
f(-v_{k,\ell_k}^{-(1-\eta)}),f(v_{k,\ell_k}^{-(1-\eta)})\}.$$ For
$\beta \in \{ -v_{k,\ell_k}^{-(1-\eta)},
v_{k,\ell_k}^{-(1-\eta)}\}$, we have \begin{eqnarray*}f(\beta
)&=&1+\frac 12\beta-\frac 18
\beta^2+O(\beta^3)\\
&&+(s_ka-1)
 \Big( 1 -\frac 12 (s_ka-1)^{-1}\beta-\frac 18
(s_ka-1)^{-2}\beta^2+O(\beta^3)\Big)\\
&=&s_ka-\frac{s_ka}{8(s_ka-1)}\beta^2+O(\beta^3)\\
&=&s_ka-\frac{s_ka}{8(s_ka-1)}v_{k,\ell_k}^{-2(1-\eta)}+O(v_{k,\ell_k}^{-3(1-\eta)}).
\end{eqnarray*}
Hence
\begin{eqnarray*}\sum_{\substack{1\le i\le k\\ 1\le j\le \ell_i}} \sqrt{u_{i,j}}
&\le& s_ka
\sqrt{v_{k,\ell_k}}-\frac{s_ka}{8(s_ka-1)}v_{k,\ell_k}^{-2(1-\eta)+1/2}+O(v_{k,\ell_k}^{-3(1-\eta)+1/2})\\
&\le &\sqrt{an} - \delta n^{2\eta -3/2}\end{eqnarray*} for all
sufficiently large integers $n$, where $\delta$ is a positive
constant. Thus, noting that $s_k^2 av_{k,\ell_k}=n$, for any
$\mathbf{u}\in U_n''$, and by \eqref{eqw2}, we have
\begin{eqnarray*}&&\sum_{\mathbf{u}\in U_n''}\prod_{\substack{1\le i\le k\\ 1\le j\le \ell_i}}
p(u_{i,j})\\
&\ll &\sum_{\mathbf{u}\in U_n''} \exp \Big( c_1
\sum_{\substack{1\le
i\le k\\ 1\le j\le \ell_i}} \sqrt{u_{i,j}}\Big)\\
&\ll &\sum_{\mathbf{u}\in U_n''}\exp \left( c_1\sqrt{an}-c_1\delta
n^{2\eta -3/2}\right)\\
&\ll &(n+1)^{\ell_1+\cdots +\ell_k-1}\exp \left(
c_1\sqrt{an}-c_1\delta n^{2\eta -3/2}\right)\\
&\ll &\exp \left( c_1\sqrt{an}-c_3n^{2\eta
-3/2}\right),\end{eqnarray*}for all sufficiently large integers
$n$, where $c_3$ is a positive constant. This completes the proof
of Lemma \ref{lem6}.
\end{proof}

Lemma \ref{lem6} deals with those $\mathbf{u}$ for which $u_{i,j}$
is far from $v_{i,j}$ for some pair $i,j$ with $i+j>2$. These
$\mathbf{u}$ contribute to $g(n)$ with the remainder. Now we deal
with all $\mathbf{u}$ for which $u_{i,j}$ is near to $v_{i,j}$ for
every pair $i,j$ with $i+j>2$.

\begin{lemma}\label{lem7} We have \begin{eqnarray*}\sum_{\mathbf{u}\in U_n'}\prod_{\substack{1\le i\le k\\ 1\le j\le \ell_i}}
p(u_{i,j}) =\frac{c_2^{\ell_1+\cdots +\ell_k}+O(n^{\eta-1})}{\prod
v_{i,j}}\sum_{\mathbf{u}\in U_n'} \exp \Big( c_1 \sum_{i,j}
\sqrt{u_{i,j}}\Big), \end{eqnarray*} where $c_1$ and $c_2$ are
given by \eqref{acc}.
\end{lemma}

\begin{proof} Recall that $$U_n'=\{ \mathbf{u}\in U_n : |u_{i,j}-v_{i,j}|<
v_{i,j}^{\eta} \text{ for all } i+j>2 \} .$$ Let $\mathbf{u}\in
U_n''$. By $\sum\limits_{i,j} s_iv_{i,j}=n$ and
$|u_{i,j}-v_{i,j}|< v_{i,j}^{\eta}$, we have
\begin{eqnarray*}|u_{1,1}-v_{1,1}|=\Big|\sum_{i+j>2}s_i(u_{i,j}-v_{i,j})\Big|
< \sum_{i+j>2}s_iv_{i,j}^\eta
<\sum_{i+j>2}s_in^{\eta}.\end{eqnarray*} Thus $u_{i,j}\to +\infty
$ as $n\to +\infty$ for $i+j\ge2$. Moreover, by $v_{i,j}=O(n )$,
we have $|u_{i,j}-v_{i,j}|\ll n^{\eta}$ for all $i,j$ with
$i+j\ge2$. Thus $u_{i,j}=O(n)$ all $i,j$ with $i+j\ge2$. Hence
$$\frac 1{u_{i,j}}=\frac 1{v_{i,j} +O(n^{\eta})}=\frac 1{v_{i,j} (1+O(n^{\eta-1}))}
=\frac 1{v_{i,j}} (1+O(n^{\eta-1})), \quad i+j\ge2$$ and
$$\sum_{\substack{1\le i\le k\\ 1\le j\le \ell_i}}
u_{i,j}^{-1/2}=O(n^{-1/2}).$$ Noting that
$$\prod_{\substack{1\le i\le k\\ 1\le j\le \ell_i}}
(1+O(n^{\eta-1})) =1+O(n^{\eta-1}),$$ from the above arguments,
\eqref{eqw1} and $\dfrac 34<\eta<\dfrac 56$, we have
\begin{eqnarray*}&&\sum_{\mathbf{u}\in U_n'}\prod_{\substack{1\le i\le k\\ 1\le j\le \ell_i}}
p(u_{i,j})\\
&=&\frac{c_2^{\ell_1+\cdots +\ell_k}+O(n^{-1/2})}{\prod
v_{i,j}}\big(1+O(n^{\eta-1})\big)\sum_{\mathbf{u}\in U_n'} \exp
\Big( c_1 \sum_{i,j}
\sqrt{u_{i,j}}\Big)\\
&=&\frac{c_2^{\ell_1+\cdots +\ell_k}+O(n^{\eta-1})}{\prod
v_{i,j}}\sum_{\mathbf{u}\in U_n'} \exp \Big( c_1 \sum_{i,j}
\sqrt{u_{i,j}}\Big). \end{eqnarray*}
\end{proof}

The following two lemmas devote to convert summations on integral
variables into integrals.

\begin{lemma}\label{lem2a} Suppose that $f$ is a function on $[a, b]$
such that $f'$ exists with $m$ zero points on $(a, b)$. Then
$$\sum_{a\le n\le b} f(n) =\int_a^b f(x) dx+O\Big( (m+1)\max_{a\le
x\le b} |f(x)|\Big).$$
\end{lemma}

\begin{proof} We divide $[a,b]$ into $m+1$ intervals
 $[a_0, a_1], [a_1, a_2],\cdots, [a_m, a_{m+1}]$
such that $f'\not= 0$ on $(a_i, a_{i+1})$. Given $0\le i\le m$.
 Without loss of generality, we assume that $f'> 0$ on
$(a_i, a_{i+1})$. For any integer $n\in (a_i+1, a_{i+1}-1)$, we
have
$$\int_{n-1}^n f(x) dx \le f(n) \le \int_n^{n+1} f(x) dx .$$
Summing  on $n$, we have
$$\Big|\sum_{a_i\le n\le a_{i+1}}
f(n) - \int_{a_i}^{a_{i+1}} f(x) dx \Big| =O \left( \max_{a_{i}\le
x\le a_{i+1}} |f(x)|\right) = O \left(\max_{a\le x\le b}
|f(x)|\right) .$$ Then
\begin{eqnarray*}\Big|\sum_{a\le n\le b} f(n) - \int_a^b f(x) dx\Big|
&\le&\sum_{i=0}^m\Big|\sum_{a_i\le n\le a_{i+1}}
f(n) - \int_{a_i}^{a_{i+1}} f(x) dx \Big|\\
&=&  O\left( (m+1) \max_{a\le x\le b} |f(x)|\right)
.\end{eqnarray*} This completes the proof.
\end{proof}

\begin{lemma}\label{lem2b}Suppose that $f$ is a function on
$$D=\{ (x_1, \dots, x_s) :a_i\le x_i\le a_i+b_i, \ 1\le i\le s \}$$
such that, for each $i$ and fixed $\{ x_1, \dots, x_s \} \setminus
\{ x_i\} $, the partial derivative  $\frac{\partial}{\partial x_i
}f $  exists with at most $m_i$ zero points on $[a_i, a_i+b_i]$.
Then
\begin{eqnarray*}&&\sum_{(n_1, \dots, n_s)\in D\cap \mathbb{Z}^s}
f(n_1, \dots, n_s)\\
& =&\underset{D}{\int\cdots\int }f(x_1, \dots, x_s) dx_1 \cdots d
x_s +O\Big( \sum_{1\le i\le s}\frac {m_i+1} {b_i}
BM\Big),\end{eqnarray*} where $B=b_1b_2\cdots b_s$ and $M=\max_D
|f(x_1, \dots, x_s)|$.
\end{lemma}

\begin{proof} Let
$c_i=a_i+b_i (1\le i\le s)$. By Lemma \ref{lem2a}, we have
\begin{eqnarray*}&&\sum_{(n_1, \dots, n_s)\in D\cap \mathbb{Z}^s}
f(n_1, \dots, n_s)\\
&=& \sum_{\substack{ a_2\le n_2\le c_2\\\cdots\\
 a_s\le n_s\le c_s}} \left( \int_{a_1}^{c_1} f(x_1, n_2, \dots, n_s) d x_1
+ O((m_1+1)M)\right)\\
&=& \sum_{\substack{ a_2\le n_2\le c_2\\\cdots\\ a_s\le n_s\le
c_s}}  \int_{a_1}^{c_1} f(x_1, n_2,\cdots, n_s) d x_1
+ O\big(\frac {m_1+1}{b_1} B M\big)\\
&=& \sum_{\substack {a_3\le n_3\le c_3\\ \cdots \\a_s\le n_s\le
c_s} }\left( \int_{a_2}^{c_2} \int_{a_1}^{c_1} f(x_1, x_2, n_3,
\dots,n_s) d x_1 d x_2 + O((m_2+1)b_1 M) \right)\\
&& \
+ O\big(\frac {m_1+1} {b_1} B M\big)\\
&=& \sum_{\substack {a_3\le n_3\le c_3\\ \cdots \\a_s\le n_s\le
c_s} } \int_{a_2}^{c_2} \int_{a_1}^{c_1} f(x_1, x_2, n_3,
\dots,n_s) d x_1 d x_2+O\Big( \big( \frac {m_1+1} {b_1}
+\frac {m_2+1} {b_2}\big) B M\Big)\\
&\vdots&\\
&=&\underset{D}{\int\cdots\int }f(x_1, \dots, x_s) dx_1 \cdots d
x_s +O\Big( \sum_{1\le i\le s}\frac {m_i+1} {b_i} B M\Big).
\end{eqnarray*} This completes the proof of Lemma
\ref{lem2b}.\end{proof}

\begin{lemma}\label{lem9} We have  \begin{eqnarray*}\sum_{\mathbf{u}\in U_n'}\prod_{\substack{1\le i\le k\\ 1\le j\le \ell_i}}
p(u_{i,j}) &=& \frac{c_2^{\ell_1+\cdots
+\ell_k}+O(n^{\eta-1})}{\prod v_{i,j}}\underset{D_1}{\int \cdots
\int } \exp \Big( c_1
\sum_{i,j} \sqrt{x_{i,j}}\Big) d x_{1,2} \cdots d x_{k,\ell_k}\\
&& +O\left(n^{-(1-\eta)(\ell_1+\cdots +\ell_k)-2\eta} \exp \left(
c_1\sqrt{an} \right) \right),
\end{eqnarray*}where $c_1$ and $c_2$ are given by \eqref{acc}, and  $D_1$ is the set of all
$\ell_1+\cdots +\ell_k -1$ tuples $(x_{1,2}, \dots ,
x_{k,\ell_k})$ of real numbers with
\begin{equation}\label{eqn1}|x_{i,j}-v_{i,j}|\le
v_{i,j}^{\eta}\end{equation} for all $i+j>2$, and
$$x_{1,1}=n-\sum_{i+j>2} s_i x_{i,j} .$$
\end{lemma}

\begin{proof} By Lemma \ref{lem7}, we have
\begin{eqnarray*}&&\sum_{\mathbf{u}\in U_n'}\prod_{\substack{1\le i\le k\\ 1\le j\le \ell_i}}
p(u_{i,j})=\frac{c_2^{\ell_1+\cdots +\ell_k}+O(n^{\eta-1})}{\prod
v_{i,j}}\sum_{\mathbf{u}\in U_n'} \exp \Big( c_1 \sum_{i,j}
\sqrt{u_{i,j}}\Big). \end{eqnarray*} For the function
$$f(x_{1,2}, \dots,
x_{1,\ell_1}\dots,x_{k,1},\dots,x_{k,\ell_k})=\exp{\Big(\sqrt{n-\sum_{i+j>2}s_ix_{i,j}}+\sum_{i+j>2}\sqrt{x_{i,j}}\Big)},$$
 by Lemma
\ref{lem2b}, we have\begin{eqnarray*}&&\sum_{\mathbf{u}\in
U_n'}\prod_{\substack{1\le i\le k\\ 1\le j\le \ell_i}}
p(u_{i,j})\\&=&\frac{c_2^{\ell_1+\cdots
+\ell_k}+O(n^{\eta-1})}{\prod v_{i,j}}\underset{D_1}{\int \cdots
\int } \exp \Big( c_1
\sum_{i,j} \sqrt{x_{i,j}}\Big) d x_{1,2} \cdots d x_{k,\ell_k}\\
&& +O\Big(n^{-(1-\eta)(\ell_1+\cdots +\ell_k)-2\eta}
\max_{D_1}\exp \Big( c_1 \sum_{i,j} \sqrt{x_{i,j}}\Big) \Big).
\end{eqnarray*}
By the Cauchy-Schwarz inequality, we
have\begin{eqnarray*}\sum_{i,j} \sqrt{x_{i,j}}=&=&\sum_{i,j}\frac
1{\sqrt{s_i}}
\sqrt{s_ix_{i,j}}\\
&\le &\Big(\sum_{i,j}\frac 1{s_i}\Big)^{1/2}
\Big(\sum_{i,j}s_ix_{i,j}\Big)^{1/2}\\
&=&\sqrt{an}.
\end{eqnarray*}
Thus,we have\begin{eqnarray*}&&O\Big(n^{-(1-\eta)(\ell_1+\cdots
+\ell_k)-2\eta} \max_{D_1}\exp \Big( c_1 \sum_{i,j}
\sqrt{x_{i,j}}\Big) \Big)\\&=&O\left(n^{-(1-\eta)(\ell_1+\cdots
+\ell_k)-2\eta} \exp \left( c_1\sqrt{an} \right)
\right).\end{eqnarray*}
\end{proof}

\begin{lemma}\label{lem8}
We have \begin{eqnarray*} g(n) &=&\frac{c_2^{\ell_1+\cdots
+\ell_k}+O(n^{3\eta -5/2})}{v_{1,1}}\left(\frac
{8}{c_1}\right)^{(\ell_1+\cdots +\ell_k -1)/2}
\left(\frac{v_{1,1}}{\prod v_{i,j}}\right)^{1/4}\exp
(c_1\sqrt{an})\nonumber\\
&&\sqrt{s_1}^{\ell_1-1} \sqrt{s_2}^{\ell_2}\cdots
\sqrt{s_k}^{\ell_k}\cdot \underset{\Omega}{\int \cdots \int } \exp
\Big( -\sum_{i,j} s_i w_{i,j}^2\Big) d w_{1,2} \cdots d
w_{k,\ell_k}\nonumber \\
&& +O\left(n^{-(1-\eta)(\ell_1+\cdots +\ell_k)-2\eta} \exp \left(
c_1\sqrt{an} \right) \right),
\end{eqnarray*}
where $\Omega$ is the set of all $\ell_1+\cdots +\ell_k -1$ tuples
$(w_{1,2}, \dots , w_{k,\ell_k})$ of real numbers with $
|w_{i,j}|< \sqrt{\dfrac{c_1}{8s_i}}v_{i,j}^{\eta -3/4}$ for all
$i+j>2$, and $ w_{1,1}=-\sum_{i+j>2}w_{i,j}$.

\end{lemma}

\begin{proof}
Noting that
$$\exp \left(
c_1\sqrt{an}-c_2n^{2\eta -3/2}\right)
=O\left(n^{-(1-\eta)(\ell_1+\cdots +\ell_k)-2\eta} \exp \left(
c_1\sqrt{an} \right) \right),$$  by \eqref{eqg(n)}, Lemmas
\ref{lem6} and \ref{lem9}, we have
\begin{eqnarray}\label{eqg(n)1} g(n) &=& \frac{c_2^{\ell_1+\cdots
+\ell_k}+O(n^{\eta-1})}{\prod v_{i,j}}\underset{D_1}{\int \cdots
\int } \exp \Big( c_1
\sum_{i,j} \sqrt{x_{i,j}}\Big) d x_{1,2} \cdots d x_{k,\ell_k}\nonumber \\
&& +O\left(n^{-(1-\eta)(\ell_1+\cdots +\ell_k)-2\eta} \exp \left(
c_1\sqrt{an} \right) \right).
\end{eqnarray}

Recall that  $D_1$ is the set of all $\ell_1+\cdots +\ell_k -1$
tuples $(x_{1,2}, \dots , x_{k,\ell_k})$ of real numbers with
\begin{equation*}|x_{i,j}-v_{i,j}|\le
v_{i,j}^{\eta}\end{equation*} for all $i+j>2$, and
$$x_{1,1}=n-\sum_{i+j>2} s_i x_{i,j} .$$

Let $x_{i,j}=v_{i,j}+v_{i,j}y_{i,j}$ for all $i,j$. Then
\begin{eqnarray*}&&\underset{D_1}{\int \cdots \int } \exp \Big( c_1 \sum_{i,j}
\sqrt{x_{i,j}}\Big)  d x_{1,2} \cdots d x_{k,\ell_k}\\
&=&\frac{\prod v_{i,j}}{v_{1,1}} \underset{D_2}{\int \cdots \int }
\exp \Big( c_1 \sum_{i,j} \sqrt{v_{i,j}+v_{i,j}y_{i,j}}\Big)  d
y_{1,2} \cdots d y_{k,\ell_k},\end{eqnarray*} where $D_2$ is the
set of all $\ell_1+\cdots +\ell_k -1$ tuples $(y_{1,2}, \dots ,
y_{k,\ell_k})$ of real numbers with
\begin{equation}\label{eqn2} |y_{i,j}|\le
v_{i,j}^{-(1-\eta)}\end{equation} for all $i+j>2$, and
\begin{equation}\label{eqn3} s_1v_{1,1}y_{1,1}=-\sum_{i+j>2} s_i v_{i,j} y_{i,j} .\end{equation}
The last equality comes from $s_1=1$ and
$$\sum_{i,j} s_ix_{i,j} =n=\sum_{i,j} s_iv_{i,j}.$$
By \eqref{eqn2} and \eqref{eqn3}, we have
\begin{equation}\label{eqn4}
|y_{i,j}|\ll n^{-(1-\eta )}\end{equation} for all $i,j$. By the
definition of $v_{i,j}$, we have
$$s_iv_{i,j}=\frac {n}{s_ia}.$$
Thus, by \eqref{eqn3}, we have
$$\sum_{i,j}\frac {n}{s_ia} y_{i,j} =0.$$
That is,
$$\sum_{i,j}\frac {y_{i,j}}{s_i}  =0.$$
Hence
$$\sum_{i,j} \sqrt{v_{i,j}} y_{i,j} =\sum_{i,j} \sqrt{\frac{n}{s_i^2
a}}y_{i,j}= \sqrt{\frac{n}{ a}}\sum_{i,j}\frac {y_{i,j}}{s_i}
=0.$$ Thus
 \begin{eqnarray*}&&\sum_{i,j} \sqrt{v_{i,j}+v_{i,j}y_{i,j}}\\
 &=&\sum_{i,j} \sqrt{v_{i,j}}\left(1+\frac 12y_{i,j}-\frac
 18y_{i,j}^2 +O(y_{i,j}^3)\right)\\
 &=&\sum_{i,j} \sqrt{v_{i,j}}-\frac 18 \sum_{i,j}
 \sqrt{v_{i,j}}y_{i,j}^2 +O(n^{3\eta -5/2})\\
&=& \sum_{i,j} \sqrt{\frac{n}{s_i^2a}}-\frac 18 \sum_{i,j}
 \sqrt{v_{i,j}}y_{i,j}^2 +O(n^{3\eta -5/2})\\
&=&\sqrt{an}-\frac 18 \sum_{i,j}
 \sqrt{v_{i,j}}y_{i,j}^2 +O(n^{3\eta -5/2}).\end{eqnarray*} Therefore
 \begin{eqnarray*}&&\underset{D_2}{\int \cdots \int }
\exp \Big( c_1 \sum_{i,j} \sqrt{v_{i,j}+v_{i,j}y_{i,j}}\Big)  d
y_{1,2} \cdots d y_{k,\ell_k}\\
&=&\exp\left(O(n^{3\eta -5/2})\right) \exp
(c_1\sqrt{an})\underset{D_2}{\int \cdots \int } \exp \Big( -\frac
{c_1}8 \sum_{i,j}
 \sqrt{v_{i,j}}y_{i,j}^2\Big)  d
y_{1,2} \cdots d y_{k,\ell_k}\\
&=&\left(1+O\left(n^{3\eta -5/2}\right)\right)\exp
(c_1\sqrt{an})\underset{D_2}{\int \cdots \int } \exp \Big( -\frac
{c_1}8 \sum_{i,j}
 \sqrt{v_{i,j}}y_{i,j}^2\Big)  d
y_{1,2} \cdots d y_{k,\ell_k}.\end{eqnarray*} Let
$$ v_{i,j}^{1/4}y_{i,j}=z_{i,j}, \quad \sqrt{\frac{c_1}{8}} z_{i,j}
=\sqrt{s_i} w_{i,j}$$ for all $i,j$. Then
\begin{eqnarray*}&&
\underset{D_2}{\int \cdots \int } \exp \left( -\frac {c_1}8
\sum_{i,j}
 \sqrt{v_{i,j}}y_{i,j}^2\right)  d
y_{1,2} \cdots d y_{k,\ell_k}\\
&=&\left(\frac{v_{1,1}}{\prod
v_{i,j}}\right)^{1/4}\underset{D_3}{\int \cdots \int } \exp \left(
-\frac {c_1}8 \sum_{i,j} z_{i,j}^2\right)  d
z_{1,2} \cdots d z_{k,\ell_k}\\
&=&\left(\frac{v_{1,1}}{\prod v_{i,j}}\right)^{1/4}\left(\frac
{8}{c_1}\right)^{(\ell_1+\cdots +\ell_k
-1)/2}\sqrt{s_1}^{\ell_1-1}
\sqrt{s_2}^{\ell_2}\cdots \sqrt{s_k}^{\ell_k}\\
&&\cdot \underset{\Omega }{\int \cdots \int } \exp \left(
-\sum_{i,j} s_i w_{i,j}^2\right) d w_{1,2} \cdots d
w_{k,\ell_k},\end{eqnarray*} where $D_3$ is the set of all
$\ell_1+\cdots +\ell_k -1$ tuples $(z_{1,2}, \dots ,
z_{k,\ell_k})$ of real numbers with
\begin{equation*}\label{eqn5} |z_{i,j}|\le
v_{i,j}^{\eta -3/4}\end{equation*} for all $i+j>2$, and
\begin{equation*}\label{eqn6} s_1v_{1,1}^{3/4}z_{1,1}=-\sum_{i+j>2} s_i v_{i,j}^{3/4} z_{i,j} ,\end{equation*}
$\Omega$ is the set of all $\ell_1+\cdots +\ell_k -1$ tuples
$(w_{1,2}, \dots , w_{k,\ell_k})$ of real numbers with
\begin{equation*}\label{eqn9} |w_{i,j}|<\sqrt{\frac{c_1}{8s_i}}v_{i,j}^{\eta -3/4}
\end{equation*} for all $i+j>2$, and
\begin{equation}\label{eqn10} s_1^{3/2}v_{1,1}^{3/4}w_{1,1}=-\sum_{i+j>2} s_i^{3/2} v_{i,j}^{3/4} w_{i,j} .\end{equation}
Noting that $s_i^{3/2} v_{i,j}^{3/4}=(n/a)^{3/4}$, \eqref{eqn10}
is equivalent to
\begin{equation}\label{eqn11} w_{1,1}=-\sum_{i+j>2}w_{i,j} .\end{equation}
Hence
\begin{eqnarray*}&&\underset{D_1}{\int \cdots \int } \exp \left( c_1 \sum_{i,j}
\sqrt{x_{i,j}}\right)  d x_{1,2} \cdots d x_{k,\ell_k}\\
&=&\frac{\prod v_{i,j}}{v_{1,1}} \underset{D_2}{\int \cdots \int }
\exp \left( c_1 \sum_{i,j} \sqrt{v_{i,j}+v_{i,j}y_{i,j}}\right)  d
y_{1,2} \cdots d y_{k,\ell_k}\\
&=&\frac{\prod v_{i,j}}{v_{1,1}}\left(1+O\left(n^{3\eta
-5/2}\right)\right)\exp (c_1\sqrt{an}) \underset{D_2}{\int \cdots
\int } \exp \left( -\frac {c_1}8 \sum_{i,j}
\sqrt{v_{i,j}}y_{i,j}^2\right)  d
y_{1,2} \cdots d y_{k,\ell_k} \\
&=&\frac{\prod v_{i,j}}{v_{1,1}}\left(1+O\left(n^{3\eta
-5/2}\right)\right) \exp (c_1\sqrt{an}) \left(\frac{v_{1,1}}{\prod
v_{i,j}}\right)^{1/4}\left(\frac {8}{c_1}\right)^{(\ell_1+\cdots
+\ell_k
-1)/2}\\
&&\sqrt{s_1}^{\ell_1-1} \sqrt{s_2}^{\ell_2}\cdots
\sqrt{s_k}^{\ell_k}\cdot \underset{\Omega}{\int \cdots \int } \exp
\left( -\sum_{i,j} s_i w_{i,j}^2\right) d w_{1,2} \cdots d
w_{k,\ell_k} .\end{eqnarray*} It follows from \eqref{eqg(n)1}and
$\dfrac 34 <\eta<\dfrac 56$ that
\begin{eqnarray*} g(n) &=& \frac{c_2^{\ell_1+\cdots
+\ell_k}+O(n^{\eta-1})}{\prod v_{i,j}}\underset{D_1}{\int \cdots
\int } \exp \left( c_1
\sum_{i,j} \sqrt{x_{i,j}}\right) d x_{1,2} \cdots d x_{k,\ell_k}\nonumber \\
&& +O\left(n^{-(1-\eta)(\ell_1+\cdots +\ell_k)-2\eta} \exp \left(
c_1\sqrt{an} \right) \right)\nonumber\\
&=& \frac{c_2^{\ell_1+\cdots +\ell_k}+O(n^{\eta-1})}{\prod
v_{i,j}}\frac{\prod v_{i,j}}{v_{1,1}} \left(1+O\left(n^{3\eta
-5/2}\right)\right) \exp (c_1\sqrt{an})
\left(\frac{v_{1,1}}{\prod v_{i,j}}\right)^{1/4}\nonumber\\
&&\left(\frac {8}{c_1}\right)^{(\ell_1+\cdots +\ell_k
-1)/2}\sqrt{s_1}^{\ell_1-1} \sqrt{s_2}^{\ell_2}\cdots
\sqrt{s_k}^{\ell_k}\nonumber\\
&&\cdot \underset{\Omega}{\int \cdots \int } \exp \left(
-\sum_{i,j} s_i w_{i,j}^2\right) d w_{1,2} \cdots d
w_{k,\ell_k}\nonumber \\
&& +O\left(n^{-(1-\eta)(\ell_1+\cdots +\ell_k)-2\eta} \exp \left(
c_1\sqrt{an} \right) \right)\\
&=&\frac{c_2^{\ell_1+\cdots +\ell_k}+O(n^{3\eta
-5/2})}{v_{1,1}}\left(\frac {8}{c_1}\right)^{(\ell_1+\cdots
+\ell_k -1)/2} \left(\frac{v_{1,1}}{\prod
v_{i,j}}\right)^{1/4}\exp
(c_1\sqrt{an})\nonumber\\
&&\sqrt{s_1}^{\ell_1-1} \sqrt{s_2}^{\ell_2}\cdots
\sqrt{s_k}^{\ell_k}\cdot \underset{\Omega}{\int \cdots \int } \exp
\left( -\sum_{i,j} s_i w_{i,j}^2\right) d w_{1,2} \cdots d
w_{k,\ell_k}\nonumber \\
&& +O\left(n^{-(1-\eta)(\ell_1+\cdots +\ell_k)-2\eta} \exp \left(
c_1\sqrt{an} \right) \right).
\end{eqnarray*}
\end{proof}

Now we determine the value of integral in Lemma \ref{lem8}
$$\underset{\Omega}{\int \cdots \int } \exp \Big( -\sum_{i,j}
s_i w_{i,j}^2\Big) d w_{1,2} \cdots d w_{k,\ell_k}.$$ To do this,
we need the following general lemma. We believe it should appear
in somewhere.
\begin{lemma}\label{lem3} If $\mathbf{x}^TA\mathbf{x}$ is a positive definite quadratic form in $\mathbf{x}^T= (x_1, \dots , x_k)$ and
$U$ is a region in $\mathbb{R}^k$, then there is a linear
transformation $\mathbf{x}=C\mathbf{y}$ such that
$$\underset{U}{\int \cdots \int } e^{- \mathbf{x}^TA\mathbf{x}} d
\mathbf{x} =\frac 1{\sqrt{\det (A)}} \underset{U'}{\int \cdots
\int } e^{-y_1^2-\cdots -y_k^2} d \mathbf{y},$$ where $$U'=\{
\mathbf{y} : \mathbf{y} =C^{-1}\mathbf{x}, \mathbf{x}\in U \} .$$
\end{lemma}

\begin{proof} Let $\mathbf{x}=C\mathbf{y}$ be a linear
transformation such that $C^TAC=I$ is the unit matrix. Then
$$|\det (C)|=\frac 1{\sqrt{\det (A)}}.$$ Therefore
\begin{eqnarray*}\underset{U'}{\int \cdots \int } e^{-\mathbf{x}^TA\mathbf{x}} d
\mathbf{x}&=&|\det (C)|\underset{U'}{\int \cdots \int }
e^{-y_1^2-\cdots -y_k^2} d \mathbf{y}\\
&=&\frac 1{\sqrt{\det (A)}} \underset{U'}{\int \cdots \int
}e^{-y_1^2-\cdots -y_k^2} d \mathbf{y}.\end{eqnarray*}
\end{proof}

\begin{lemma}\label{lem4} Let $a_0, a_1, \dots , a_k$ be $k+1$ positive
real numbers. Then there is a linear transformation
$\mathbf{x}=C\mathbf{y}$ such that
\begin{eqnarray*}&&\underset{U}{\int \cdots \int
} \exp \left(-a_0 (x_1+\cdots +x_k)^2 -
a_1x_1^2-\cdots -a_kx_k^2 \right) d \mathbf{x} \\
&=&\left( a_0a_1\cdots a_k \sum_{i=0}^k \frac 1{a_i}
\right)^{-1/2}\underset{U'}{\int \cdots \int }e^{-y_1^2-\cdots
-y_k^2} d \mathbf{y},\end{eqnarray*}  where $$U'=\{ \mathbf{y} :
\mathbf{y} =C^{-1}\mathbf{x}, \mathbf{x}\in U \} .$$
\end{lemma}

\begin{proof} It is clear that the quadratic form
$$a_0 (x_1+\cdots +x_k)^2 +a_1x_1^2+\cdots +a_kx_k^2$$ is positive definite. Its matrix is
$$A_k=\left(\begin{matrix}&a_0+a_1 &a_0 &a_0  &\cdots  &a_0\\
&a_0 &a_0+a_2  &a_0  &\cdots  &a_0\\
&\cdot  &\cdot &\cdot &\cdots &\cdot\\
&\cdot  &\cdot &\cdot &\cdots &\cdot\\
&\cdot  &\cdot &\cdot &\cdots &\cdot\\
&a_0 &a_0  &a_0  &\cdots &a_0+a_k
\end{matrix}\right).
$$
It is not difficult to see that $$ \det(A_k)=a_k
\det(A_{k-1})+a_0a_1\cdots a_{k-1}.$$ By induction on $k$, we have
$$\det (A_k) =a_0a_1\cdots a_k \sum_{i=0}^k \frac 1{a_i}.$$
Now Lemma \ref{lem4} follows from Lemma \ref{lem3}.
\end{proof}

\section{Proof of Theorem \ref{thm0} }
It follows from Lemma \ref{lem8} that
\begin{eqnarray*} g(n) &=&\frac{c_2^{\ell_1+\cdots
+\ell_k}+O(n^{3\eta -5/2})}{v_{1,1}}\left(\frac
{8}{c_1}\right)^{(\ell_1+\cdots +\ell_k -1)/2}
\left(\frac{v_{1,1}}{\prod v_{i,j}}\right)^{1/4}\exp
(c_1\sqrt{an})\nonumber\\
&&\sqrt{s_1}^{\ell_1-1} \sqrt{s_2}^{\ell_2}\cdots
\sqrt{s_k}^{\ell_k}\cdot \underset{\Omega}{\int \cdots \int } \exp
\Big( -\sum_{i,j} s_i w_{i,j}^2\Big) d w_{1,2} \cdots d
w_{k,\ell_k}\nonumber \\
&& +O\left(n^{-(1-\eta)(\ell_1+\cdots +\ell_k)-2\eta} \exp \left(
c_1\sqrt{an} \right) \right) \end{eqnarray*} Thus, we only need
estimate $\underset{\Omega}{\int \cdots \int } \exp \Big(
-\sum_{i,j} s_i w_{i,j}^2\Big) d w_{1,2} \cdots d w_{k,\ell_k}$.
By Lemma \ref{lem4}, we have
\begin{eqnarray*}&&\underset{\Omega}{\int \cdots \int } \exp \Big( -\sum_{i,j}
s_i w_{i,j}^2\Big) d w_{1,2} \cdots d w_{k,\ell_k}\\
&=&\left( s_1^{\ell_1} s_2^{\ell_2}\cdots s_k^{\ell_k}
\sum_{i=1}^k \frac{\ell_i}{s_i}
\right)^{-1/2}\underset{\Omega'}{\int \cdots \int }\exp \Big(
-\sum_{i+j>2} w_{i,j}'^2\Big) d w_{1,2} '\cdots d
w_{k,\ell_k}',\end{eqnarray*} where $\Omega'$ is the set of all
$\ell_1+\cdots +\ell_k -1$ tuples $(w_{1,2}', \dots ,
w_{k,\ell_k}')$ of real numbers with
\begin{equation*}\label{eqn9} |w_{i,j}'|<c_{i,j}'v_{i,j}^{\eta -3/4}
\end{equation*} for all $i+j>2$ and some positive constants $c_{i,j}'$. Noting that, for $x>1$, $$
\int_x^\infty e^{-t^2}dt=\int_x^\infty e^{-t^2+t-t}dt\le
e^{-x^2+x} \int_x^\infty e^{-t}dt = e^{-x^2}
$$ and the well known Gaussian integral (also known as the Euler
Poisson integral)
$$\int_{-\infty}^{+\infty} e^{-x^2} dx=\sqrt{\pi},$$  for $\eta>\dfrac
34$, we have\begin{eqnarray*}\int_{-c_{i,j}'v_{i,j}^{\eta
-3/4}}^{c_{i,j}'v_{i,j}^{\eta -3/4}} e^{-x^2}dx
&=&\int_{-c_{i,j}''n^{\eta -3/4}}^{c_{i,j}''n^{\eta -3/4}}e^{-x^2}dx\\
&=&\int_{-\infty}^\infty e^{-x^2}dx + O(e^{-(c_{i,j}''^2n^{2\eta-3/2})})\\
&=&\sqrt{\pi}+  O(e^{-(c_{i,j}''^2n^{2\eta
-3/2})}),\end{eqnarray*}where $c_{i,j}''$ are some positive
constants. Thus
\begin{eqnarray*}&&\underset{\Omega'}{\int \cdots \int }\exp \Big(
-\sum_{i+j>2} w_{i,j}'^2\Big) d w_{1,2} '\cdots d w_{k,\ell_k}'\\
&=&\left(\sqrt{\pi}+  O\left(e^{-(c_{1,2}''^2n^{2\eta
-3/2})}\right)\right)\cdots\left(\sqrt{\pi}+  O\left(e^{-(c_{k,\ell_k}''^2n^{2\eta-3/2})}\right)\right)\\
&=&\pi^{(\ell_1 +\cdots +\ell_k -1)/2}+O\left(e^{-(c
n^{2\eta-3/2})}\right),\end{eqnarray*}where $c$ is a positive
constant. Therefore,\begin{eqnarray*}
g(n)&=&\frac{c_2^{\ell_1+\cdots +\ell_k}+O(n^{3\eta
-5/2})}{v_{1,1}}\left(\frac {8}{c_1}\right)^{(\ell_1+\cdots
+\ell_k -1)/2} \left(\frac{v_{1,1}}{\prod
v_{i,j}}\right)^{1/4}\exp
(c_1\sqrt{an})\nonumber\\
&&\cdot \sqrt{s_1}^{\ell_1-1} \sqrt{s_2}^{\ell_2}\cdots
\sqrt{s_k}^{\ell_k}\cdot \left( s_1^{\ell_1} s_2^{\ell_2}\cdots
s_k^{\ell_k} \sum_{i=1}^k \frac{\ell_i}{s_i} \right)^{-1/2}\\
&& \cdot \big(\pi^{(\ell_1 +\cdots +\ell_k -1)/2}+O(e^{-(c
n^{2\eta-3/2})})\big)\nonumber
\\
&& +O\left(n^{-(1-\eta)(\ell_1+\cdots +\ell_k)-2\eta} \exp \left(
c_1\sqrt{an} \right) \right)\\
&=&\frac{c_2^{\ell_1+\cdots +\ell_k}\pi^{(\ell_1 +\cdots+\ell_k
-1)/2} +O(n^{3\eta -5/2})}{v_{1,1}}\left(\frac
{8}{c_1}\right)^{(\ell_1+\cdots +\ell_k -1)/2}
\left(\frac{v_{1,1}}{\prod v_{i,j}}\right)^{1/4}\nonumber\\
&&\exp (c_1\sqrt{an})\sqrt{s_1}^{\ell_1-1}
\sqrt{s_2}^{\ell_2}\cdots \sqrt{s_k}^{\ell_k}\cdot \left(
s_1^{\ell_1} s_2^{\ell_2}\cdots
s_k^{\ell_k} \sum_{i=1}^k \frac{\ell_i}{s_i} \right)^{-1/2}\nonumber \\
&& +O\left(n^{-(1-\eta)(\ell_1+\cdots +\ell_k)-2\eta} \exp \left(
c_1\sqrt{an} \right) \right)\\
&=&2^{-(3\ell_1+\cdots +3\ell_k+5)/4} 3^{-(\ell_1+\cdots
+\ell_k+1)/4} a^{(\ell_1+\cdots +\ell_k+1)/4} s_1^{\ell_1/2}\cdots
s_k^{\ell_k /2}\\
&& \cdot n^{-3/4 - (\ell_1+\cdots +\ell_k)/4} \exp
(c_1\sqrt{an})\\
&&+O\left(n^{3\eta-\frac 52-1-(\ell_1+\cdots +\ell_k-1)/4}\exp
\left(c_1\sqrt{an}
\right)\right)\\
&&+O\left(n^{-(1-\eta)(\ell_1+\cdots +\ell_k)-2\eta} \exp
\left(c_1\sqrt{an} \right) \right)\\
&=&c(\mathbf{s}, \mathbf{l}) n^{d(\mathbf{l} )}\exp
\left(c_1\sqrt{an} \right) \\
&&+O\left(n^{3\eta-\frac 72-(\ell_1+\cdots +\ell_k-1)/4}\exp
\left(c_1\sqrt{an}
\right)\right)\\
&&+O\left(n^{-(1-\eta)(\ell_1+\cdots +\ell_k)-2\eta} \exp
\left(c_1\sqrt{an} \right) \right).
\end{eqnarray*}
Noting that $\eta$ is a real number that subjects to $$\frac 34
<\eta <\min\left\{\frac 56,\ \frac 34 \frac{\ell_1+\cdots
+\ell_k-1}{\ell_1+\cdots +\ell_k-2}\right\},$$ we choose
$\eta=\dfrac 34+\varepsilon'$, where $\varepsilon'$ is
sufficiently small positive real number. Thus
\begin{eqnarray*} &&O(n^{3\eta-\frac 72-(\ell_1+\cdots
+\ell_k-1)/4}\exp \left(c_1\sqrt{an}
\right))+O\left(n^{-(1-\eta)(\ell_1+\cdots +\ell_k)-2\eta} \exp
\left(c_1\sqrt{an} \right) \right)\\
&=&O\left(n^{-1+\varepsilon-(\ell_1+\cdots +\ell_k)/4}\exp
\left(c_1\sqrt{an} \right)\right)\\
&=& O\left( n^{d(\mathbf{l} )-\frac 14+\varepsilon }
\exp\left(c_1\sqrt{an} \right)\right) .\end{eqnarray*} Noting that
$$\exp\left(c_1\sqrt{an} \right) =\exp \left(\pi \sqrt{\frac{2a
(\mathbf{s}, \mathbf{l}) n}3}\right),$$ we obtain a proof of
Theorem \ref{thm0}.

\section*{Acknowledgments}

This work was supported by the National Natural Science Foundation
of China, Grant No. 11371195 and a project funded by the Priority
Academic Program Development of Jiangsu Higher Education
Institutions.

\end{document}